\documentclass{amsart}
\usepackage{amsmath}
\usepackage{amssymb}
\usepackage{amsthm}
\usepackage{enumerate}
\usepackage[dvipdfmx]{graphicx}
\theoremstyle{definition}
\newtheorem{definition}{Definition}[section]
\theoremstyle{plain}
\newtheorem{lemma}[definition]{Lemma}

\newtheorem{theorem}[definition]{Theorem}

\theoremstyle{remark}
\newtheorem{remark}[definition]{Remark}

\newtheorem{example}[definition]{Example}

\newcommand{\myIso}{\operatorname{iso}}
\newcommand{\mycl}{\operatorname{cl}}
\newcommand{\myfr}{\operatorname{fr}}
\newcommand{\myadh}{\operatorname{Adh}}
\newcommand{\mycon}{\operatorname{Con}}
\newcommand{\myreg}{\operatorname{Reg}}
\newcommand{\myint}{\operatorname{int}}
\newcommand{\myrank}{\operatorname{rank}}

\begin{document}
	
	\title[Connected components in d-minimal structures]{Connected components in d-minimal structures}
	\author[M. Fujita]{Masato Fujita}
	\address{Department of Liberal Arts,
		Japan Coast Guard Academy,
		5-1 Wakaba-cho, Kure, Hiroshima 737-8512, Japan}
	\email{fujita.masato.p34@kyoto-u.jp}
	
	\begin{abstract}
		For a given d-minimal expansion $\mathfrak R$ of the ordered real field, we consider the expansion $\mathfrak R^\natural$ of $\mathfrak R$ generated by the sets of the form $\bigcup_{S \in \mathcal C}S$, where $\mathcal C$ is a subfamily of the collection of connected components of an $\mathfrak R$-definable set.
		We prove that $\mathfrak R^{\natural}$ is d-minimal.
		A similar assertion holds for almost o-minimal expansions of ordered groups.
	\end{abstract}
	
	
	
	\thanks{The author was supported by JSPS KAKENHI Grant Number JP25K07109. He is grateful to Chris Miller for helpful discussions on this topic.}
	
	\maketitle	

\section{Introduction}\label{sec:intro}
Every connected component of a set definable in an o-minimal structure over $\mathbb R$ is definable, but this is not the case for locally o-minimal structures and d-minimal structures over $\mathbb R$.
For instance, \cite[Theorem 1.9]{DMST} gives a necessary and sufficient condition for every connected component of a definable set to be definable for locally o-minimal expansions of $(\mathbb R,<,+,\mathbb Z)$.
$(\mathbb R,<,+,\mathbb Z)$ has a definable set whose connected component is not definable.
The structure $(\mathbb R,+,\cdot,2^{\mathbb Z})$ is a d-minimal structure by \cite{vdD_dmin}, but there exists a set definable in this structure one of whose connected component is not definable \cite[Example 79]{Sav}.
A question is whether a given tame structure $\mathfrak R$ can be expanded to a structure $\mathfrak R^\natural$ so that every connected component of a set definable in $\mathfrak R$ is definable in $\mathfrak R^\natural$ without loosing the given tameness such as local o-minimality and d-minimality.

A d-minimal structure \cite{Miller-dmin} is a generalization of o-minimal structures.
An expansion $\mathfrak R$ of $(\mathbb R,<)$ is \textit{d-minimal} if, for every $n \geq 0$ and every definable subset $X$ of $\mathbb R^{n+1}$, there exists $N \in \mathbb N$ such that the fiber $X_x:=\{y \in \mathbb R\;|\; (x,y) \in X\}$ is a union of an open set and at most $N$ discrete sets.
Our main theorem illustrates how to construct a structure $\mathfrak R^\natural$ satisfying the above condition when $\mathfrak R$ is a d-minimal expansion of an ordered field.
The strategy of the proof is applicable to a slightly different setting.
An expansion $\mathfrak R$ of dense linear order without endpoints is \textit{almost o-minimal} if every bounded unary definable set is a union of a finite set and finitely many open intervals \cite{Fuji_almost}.
As we explain in Remark \ref{rem:almost}, the structure $\mathfrak R^\natural$ constructed in the same manner is almost o-minimal if $\mathfrak R$ is almost o-minimal even if the underlying space is not necessarily $\mathbb R$.

For every topological space $T$, $\mycon(T)$ denotes the collection of connected components of $T$.
For every expansion $\mathfrak R$ of $(\mathbb R,<)$, $\mathfrak R^\natural$ denotes the expansion of $\mathfrak R$ generated by all the sets of the form $\bigcup_{S \in \mathcal C}S$ for an $n \in \mathbb N$, an $\mathfrak R$-definable subset $X$ of $\mathbb R^n$ and a subfamily $\mathcal C$ of $\mycon(X)$.

Our main theorem is as follows:
\begin{theorem}\label{thm:connected_expansion}
For every d-minimal expansion $\mathfrak R$ of the ordered field of reals, $\mathfrak R^\natural$ is d-minimal.
Furthermore, a subset of $\mathbb R^n$ is $\mathfrak R^\natural$-definable if and only if it is a union of finitely many sets of the form $\bigcup_{S \in \mathcal C}S$, where $\mathcal C$ is a subfamily of $\mycon(X)$ for an $\mathfrak R$-definable subset $X$ of $\mathbb R^n$.
In particular, for every $\mathfrak R^\natural$-definable set $X$ and a subfamily $\mathcal C$ of $\mycon(X)$, $\bigcup_{S \in \mathcal C}S$ is $\mathfrak R^\natural$-definable. 
\end{theorem}

\section{Multi-cell decomposition in $\mathfrak R$}\label{sec:multi-cell}
We prove a multi-cell decomposition theorem, which refines partition of definable sets into special submanifolds in d-minimal structures \cite{Miller-dmin, T}.
This theorem is used in the proof of a generalized version of Theorem \ref{thm:connected_expansion}.
Throughout this section, suppose that $\mathfrak R=(\mathbb R,<,+,\ldots)$ is a d-minimal expansion of the ordered group of reals.
`Definable' means `definable in $\mathfrak R$ with parameters'.
For a subset $S$ of $\mathbb R^n$, $\myint(S)$, $\mycl(S)$ and $\myfr(S):=\mycl(S) \setminus S$ denote the interior, closure and frontier of $S$, respectively.
%
Let $\pi:\mathbb R^n \to \mathbb R^{n-1}$ be the coordinate projection forgetting the last coordinate.
For a subset $S$ of $\mathbb R^n$, $\pi|_{S}$ denotes the restriction of $\pi$ to $S$.

A subset $M$ of $\mathbb R^n$ is called a \textit{submanifold} of dimension $d$ if, for every point $x$ in $M$, there exists a homeomorphism
$\varphi:U \to V$ from an open neighborhood $U$ of $x$ in $\mathbb R^n$ onto an open
neighborhood $V$ of the origin in $\mathbb R^n$ such that $\varphi(x)$ is the origin and the equality $\varphi(M \cap U)=V \cap (\mathbb R^d \times \{\overline{0}_{n-d}\})$ holds, 
where $\overline{0}_{n-d}$ denotes the origin of $\mathbb R^{n-d}$.
We say that a submanifold $M$ is a \textit{definable submanifold} if $M$ is definable and $U$, $V$ and $\varphi$ appeared in the above definition are definable.
Observe that a definable open set is a definable submanifold.

The following definition is quoted from \cite{Fuji_almost} and slightly modified.
\begin{definition}\label{def:multicell}
	A definable submanifold $X$ of $\mathbb R^n$ is a \textit{multi-cell} if it satisfies the following conditions:
	\begin{enumerate}
		\item[(a)] If $n=1$, either $X$ is a discrete definable set or all connected components of the definable set $X$ are open intervals. 
		\item[(b)] If $n>1$, the projection image $\pi(X)$ is a multi-cell and, for any connected component $Y$ of $X$, $\pi(Y)$ is a connected component of $\pi(X)$ and $Y$ is one of the following forms:
		\begin{align*}
			Y&=\pi(Y) \times \mathbb R  \text{,}\\
			Y&=\{(x,y) \in \pi(Y) \times \mathbb R \;|\; y=f(x)\} \text{,}\\
			Y &= \{(x,y) \in \pi(Y) \times \mathbb R \;|\; y>f(x)\} \text{,}\\
			Y &= \{(x,y) \in \pi(Y) \times \mathbb R \;|\; y<g(x)\} \text{ and }\\
			Y &= \{(x,y) \in \pi(Y) \times \mathbb R \;|\; f(x)<y<g(x)\}
		\end{align*}
		for some continuous functions $f$ and $g$ defined on $\pi(Y)$ with $f<g$.
	\end{enumerate} 
	
	Let $\{A_1, \ldots, A_k\}$ be a nonempty finite family of definable subsets of $\mathbb R^n$.
	A \textit{multi-cell decomposition} partitioning $\{A_1, \ldots, A_k\}$ is a family $\mathcal C$ of finitely many multi-cells satisfying the following conditions:
	\begin{itemize}
		\item For every $1 \leq i \leq k$, $A_i$ is the union of multi-cells in a subfamily of $\mathcal C$;
		\item If $n>1$, for every $C_1,C_2 \in \mathcal C$, either $\pi(C_1)=\pi(C_2)$ or $\pi(C_1) \cap \pi(C_2)=\emptyset$ holds;
		\item If $n>1$, the collection $\pi(\mathcal C):=\{\pi(C)\;|\;C \in \mathcal C\}$ is a multi-cell decomposition partitioning $\{\pi(A_1), \ldots, \pi(A_k)\}$.
	\end{itemize}
\end{definition}

A definable submanifold $M$ of $\mathbb R^n$ is \textit{pseudo-special} if $\pi(M)$ is a submanifold of $\mathbb R^{n-1}$ and, for every $x \in M$, there exists an open box $B$ in $\mathbb R^n$ containing $x$ such that $\pi|_{B \cap M}$ is a homeomorphism onto $\pi(B) \cap \pi(M)$.
\begin{example}
	In every expansion of the ordered field of reals, the circle $S^1:=\{(x,y) \in \mathbb R^2\;|\; x^2+y^2=1\}$ is a definable submanifold, but it is not pseudo-special.
	It is easy to see that a multi-cell $C$ is a a pseudo-special submanifold if, for every $x \in \mathbb R^{n-1}$, the fiber $C_x:=\{y \in \mathbb R\;|\;(x,y) \in C\}$ has an empty interior.
	The definable submanifold $M=\{(x,y)\;|\; xy=1 \text{ or } y=0\}$ is pseudo-special, but it is not a multi-cell.
\end{example}

Consider the following property of pseudo-special submanifolds $M$ of $\mathbb R^n$
\begin{enumerate}
	\item[($\dagger$)] Let $X$ be an arbitrary connected component of $M$.
	Let $(a_m)_{m \in \mathbb N}$ be a sequence in $\pi(X)$ converging to a point $a \in \pi(M)$. 
	For every $m \in \mathbb N$, let $b_m \in X$ such that $\pi(b_m)=a_m$.
	Then, $\{b_m\;|\; m \in \mathbb N\}$ is bounded.
\end{enumerate}
Observe that every bounded pseudo-special submanifolds enjoys property $(\dagger)$.

\begin{theorem}\label{thm:decomposition_into_multicells}
	Let $\{A_1 \ldots, A_k\}$ be a nonempty finite family of definable subsets of $\mathbb R^n$.
	There exists a multi-cell decomposition partitioning $\{A_1 \ldots, A_k\}$ if one of the following conditions is satisfied:
	\begin{enumerate}
		\item[(a)] $A_i$ is bounded for every $1 \leq i \leq k$;
		\item[(b)] For every $n \in \mathbb N$ and every pseudo-special submanifold $M$ of $\mathbb R^n$, property $(\dagger)$ is fulfilled.
	\end{enumerate}
\end{theorem}

The rest of this section is devoted to a proof of Theorem \ref{thm:decomposition_into_multicells}.
This proof is a mixture of Thamrongthanyalak's proof of \cite[Theorem B]{T} and the author's proof of \cite[Theorem 4.22]{Fuji_almost}.
As an intermediate step, we want to find a sufficient condition for a pseudo-special manifold to be a multi-cell.
For that purpose, we define several technical conditions and give examples which illustrate why these conditions are necessary. 
Our intermediate target (Lemma \ref{lem:main_lemma}) is to show that a pseudo-special submanifold is a multi-cell if $\pi(M)$ is a multi-cell and $M$ satisfies these technical conditions. 

Now we start to introduce technical definitions and notations.
Let $S$ be a closed subset of a topological space. 
Let $\myIso(S)$ be the set of isolated points in $S$.
For every nonnegative integer $m$, we define $\myIso_m(S)$ as follows:
\begin{itemize}
	\item $\myIso_0(S)=\myIso(S)$;
	\item $\myIso_m(S)=\myIso(S \setminus \bigcup_{i=0}^{m-1}\myIso_i(S))$;
\end{itemize}
Define $\myrank(S)$ as the minimum of positive integers $m$ such that $\myIso_m(S)=\emptyset$.
$\myrank(S)$ is called the \textit{Cantor-Bendixson rank} of $S$.
We put $\myrank(S)=\infty$ if such an $m$ does not exist.

A subset $C$ of $\mathbb R^n$ is a \textit{connected graph} on a connected subset $V$ of $\mathbb R^{n-1}$ if $C$ is the graph of a continuous function $f_C:V \to \mathbb R$.
Throughout, for every connected graph $C$, $f_C$ denotes the continuous function whose graph is $C$.
For a connected graph $C$ on $V$, we put 
\begin{align*}
	&C^+:=\{(x,t) \in V \times \mathbb R\;|\; t>f_C(x)\} \text { and }\\
	&C^-:=\{(x,t) \in V \times \mathbb R\;|\; t<f_C(x)\}. 
\end{align*}
Observe that $C^+$ and $C^-$ are connected.
Let $X$ be a connected subset of $\mathbb R^n$ with $\pi(X) \subseteq V$ and $X \cap C =\emptyset$.
We do not require that $X$ is a connected graph.
We have either $X \subseteq C^+$ or $X \subseteq C^-$ because $X \cap C^+$ and $X \cap C^-$ are closed and open in the connected set $X$ .
We denote 
\begin{center}
	$X>_V\underline{C}$ if $X \subseteq C^+$ and $X<_V\underline{C}$ if $X \subseteq C^-$.
\end{center}
Underlines below $C$ are used to clarify which is a connected graph.

Let $C_1, C_2 \subseteq \mathbb R^n$ be two connected graphs on $V$ with $\underline{C_1}<_V \underline{C_2}$.
We denote $\{(x,t) \in V \times \mathbb R\;|\; f_{C_1}(x)<t<f_{C_2}(x)\}$ by $(C_1,C_2)_V$ and $\{(x,t) \in V \times \mathbb R\;|\; f_{C_1}(x) \leq t \leq f_{C_2}(x)\}$ by $[C_1,C_2]_V$.

\begin{lemma}\label{lem:easy}
	Let $V \subseteq \mathbb R^{m-1}$ be a submanifold and $C_1,C_2 \subseteq \mathbb R^m$ be connected graphs on $V$.
	Let $X$ be a connected subset of $\mathbb R^n$ such that $\pi(X)$ is a proper open subset of $V$ and $\underline{C_1}<_VX<_V\underline{C_2}$.
	Then $\myfr(\pi(X)) \cap V \neq \emptyset$, and, for every $x \in \myfr(\pi(X)) \cap V$, there exists $y \in \myfr(X) \cap [C_1,C_2]_V$ with $\pi(y)=x$.
\end{lemma}
\begin{proof}
	Connectivity of $V$ and openness of $\pi(X)$ in $V$ imply $\myfr(\pi(X)) \cap V \neq \emptyset$.
	Fix an arbitrary $x \in \myfr(\pi(X)) \cap V$.
	Replacing $V$ with a small neighborhood of $x$ in $V$, we may assume that $V$ is bounded.
	Take a sequence $\{x_m\}_{m \in \mathbb N}$ in $\pi(X) \cap V$ converging to $x$.
	Let $c_m \in \mathbb R$ with $(x_m,c_m) \in X$.
	The condition $\underline{C_1}<_VX<_V\underline{C_2}$ implies that $X$ is bounded. 
	Since $X$ is bounded, we may assume that $c:=\lim_{m \to \infty}c_m$ exists by taking a subsequence.
	We have $y:=(x,c) \in \myfr(X)$.
	We have $f_{C_1}(x) \leq c \leq f_{C_2}(x)$ because $\underline{C_1}<_VX<_V\underline{C_2}$.
	This means $y \in [C_1,C_2]_V$.
\end{proof}

Even if $M$ is simultaneously a multi-cell and a pseudo-special submanifold, $\myfr(M)$ does not necessarily empty as the following example illustrates. 
\begin{example}\label{ex:pair}
	Recall that $(\mathbb R,+,\cdot,2^{\mathbb Z})$ is d-minimal.
	$M:=\{x>0\} \times 2^{\mathbb Z} \subseteq \mathbb R^2$ is a multi-cell, and it is simultaneously a pseudo-special submanifold.
	We have $\myfr(M)=\{x>0\} \times \{0\}$.
\end{example}

Let $M$ be a multi-cell which is simultaneously a pseudo-special submanifold.
Put $$P_i:=\bigcup_{x \in \pi(M)} \myIso_i(\mycl(M) \cap \pi^{-1}(x)).$$
The pair $(M,P_i)$ is expected to satisfy some tame condition.
This is the reason why we introduce the following notion called well-ordered pairs.

Let $M$ and $P$ be definable subsets of $\mathbb R^n$.
We say that $(M,P)$ is a \textit{well-ordered pair} if the following conditions are satisfied:
\begin{enumerate}
	\item[(a)] $M$ is a pseudo-special submanifold of $\mathbb R^n$.
	\item[(b)] $\pi(P)=\pi(M)$;
	\item[(c)] $P \subseteq \myfr(M) \cap \pi^{-1}(\pi(M))$;
	\item[(d)] $P$ is a multi-cell.
\end{enumerate}

\begin{example}\label{ex:pair2}
	Let $M$ be as in Example \ref{ex:pair}.
	For every $x>0$, the fiber $M_x:=\{y \in \mathbb R\;|\; (x,y) \in M\}$ is of Cantor-Bendixson rank $2$.
	We have $P_1:=\bigcup_{x>0}\myIso_1(\mycl(M)_x)=\myfr(M)$.
	The pair $(M,P_1)$ is a well-ordered pair.
\end{example}

\begin{lemma}\label{lem:easyeasy}
For a well-ordered pair $(M,P)$, every connected component of the multi-cell $P$ is a connected graph.
\end{lemma}
\begin{proof}
	Let $Q$ be a connected component of $P$.
	By condition (d), $P$ is a multi-cell.
	$V:=\pi(Q)$ is a connected component of $\pi(P)$ by the definition of multi-cells.
	By condition (b), $V$ is a connected component of $\pi(M)$.
	Assume for contradiction that $Q$ is not the graph of a continuous function defined on $V$.
	By Definition \ref{def:multicell}(b), 
	$Q$ is open in $V \times \mathbb R$.
	Since $V$ is a connected component of $\pi(M)$, $Q$ is open in $\pi(M) \times \mathbb R$.
	On the other hand, condition (c) implies, that $P \cap M=\emptyset$ and, for any point $q \in Q$, we can choose a sequence $(c_n)$ in $M$ converging to $q$, which contradicts that $Q$ is open in $\pi(M) \times \mathbb R$.
	
	$Q$ is a connected graph because it is the graph of a continuous function defined on the connected set $V$.
\end{proof}

The following example illustrates that $M$ is not necessarily a multi-cell even if the fiber $M_x$ is of Cantor-Bendixson rank $2$ and $(M,P_1)$ is a well-ordered pair.
\begin{example}\label{ex:pair3}
	Recall that $(\mathbb R,+,\cdot,2^{\mathbb Z})$ is d-minimal.
	Let $M:=\{(x,2^kx)\;|\;x>0, k \in \mathbb Z\} \cup \{(x,-2^k)\;|\;x \in \mathbb R, k \in \mathbb Z\}$.
	Let $P_1:=\myfr(M)=\mathbb R \times \{0\}$.
	The pair $(M,P)$ is a well-ordered pair, but $M$ is not a multi-cell.
\end{example}
We introduce the notion of extremely well-ordered pairs in order to exclude the case given in Example \ref{ex:pair3}.

For a well-ordered pair $(M,P)$, every $C \in \mycon(P)$ is a connected graph by Lemma \ref{lem:easyeasy}.
Let $\myadh^+(M,P)$ be the set of points $z \in P$ such that $z \in \myfr(M \cap C^+)$, where $C \in \mycon(P)$ with $z \in C$. 
We define $\myadh^-(M,P)$ in the same manner.	
For every $x \in \pi(M)$, put $$\myadh^+(M,P,x):=\{x\} \times \{t \in \mathbb R\;|\; (x,t) \in P \cap \myfr(M \cap (\{x\} \times (t,\infty)))\}.$$
We define $\myadh^-(M,P,x)$ similarly.	
For every $x \in \pi(M)$, we have $\myadh^+(M,P,x) \subseteq \myadh^+(M,P) \cap \pi^{-1}(x)$ and $\myadh^-(M,P,x) \subseteq \myadh^-(M,P) \cap \pi^{-1}(x)$.

A well-ordered pair $(M,P)$ is an \textit{extremely well-ordered pair} if $$\myadh^+(M,P) \cap \pi^{-1}(x)=\myadh^+(M,P,x)$$ and $$\myadh^-(M,P) \cap \pi^{-1}(x)=\myadh^-(M,P,x)$$ for every $x \in \pi(M)$.
\begin{example}\label{ex:pair4}
	The pair $(M,P)$ in Example \ref{ex:pair3} is not an extremely well-ordered pair.
\end{example}

The following lemma guarantees that a well-ordered pair $(M,P)$ becomes an extremely well-ordered pair after removing a small definable set.
\begin{lemma}\label{lem:well_ordered}
Let $(M,P)$ be a well-ordered pair.
Let $x_0 \in \pi(M)$.
Then, $\myadh^+(M,P)$, $\myadh^-(M,P)$, $\myadh^+(M,P,x_0)$ and $\myadh^-(M,P,x_0)$ are definable.
Furthermore, $$\{x \in \pi(M)\;|\; \myadh^+(M,P) \cap \pi^{-1}(x) \neq \myadh^+(M,P,x)\}$$ and $$\{x \in \pi(M)\;|\; \myadh^-(M,P) \cap \pi^{-1}(x) \neq \myadh^-(M,P,x)\}$$ are nowhere dense in $\pi(M)$.
\end{lemma}
\begin{proof}
	We have 
	\begin{align*}
	&\myadh^+(M,P)=\{(x,y) \in \mathbb R^{n-1} \times \mathbb R\;|\; (x,y) \in P \wedge (\forall B: \text{open box with }(x,y) \in B \\
	&\quad ((B \cap P \text{ is the graph of a continuous function }f \text{ defined on } \pi(P) \cap \pi(B))\\
	&\quad  \rightarrow \exists (x_0,y_0) \in M \cap B \  (x_0 \in \pi(P) \cap \pi(B) \wedge y_0>f(x_0))))\}.
	\end{align*}
	This implies definability of $\myadh^+(M,P)$.
	$\myadh^-(M,P)$ is definable for a similar reason.
	Definability of $\myadh^+(M,P,x_0)$ and $\myadh^-(M,P,x_0)$ is much easier.
	
	We show that $S:=\{x \in \pi(M)\;|\; \myadh^+(M,P) \cap \pi^{-1}(x) \neq \myadh^+(M,P,x)\}$ is nowhere dense in $\pi(M)$.
	For every $C \in \mycon(P)$, let $f_C:\pi(C) \to \mathbb R$ be the continuous function whose graph is $C$.
	Let $g_C:\pi(C) \to \mathbb R$ be the function given by $g_C(x):=\inf\{t \in \mathbb R\;|\; t>f_C(x),\ (x,t) \in M\}$
	and 
	$$T_C:=\{y\in \myfr(M \cap C^+) \cap P\;|\; f_C(\pi(y)) < g_C(\pi(y))\}.$$
	Observe that $T_C \subseteq C$ and $S=\bigcup_{C \in \mycon(P)}\pi(T_C)$.
	
	Assume for contradiction that $S$ is somewhere dense in $\pi(M)$.
	Since $S$ is definable in the d-minimal structure $\mathfrak R$, $S$ has a nonempty interior in $\pi(M)$ by \cite[Main Lemma]{Miller-dmin}.
	Observe that $P$ has only countably many connected components.
	By the Baire category theorem, there exists $C \in \mycon(P)$ such that $\pi(T_C)$ is somewhere dense in $\pi(M)$.
	There exists a nonempty definable open subset $W$ of $\pi(M)$ such that $W$ is contained in the closure $\mycl_{\pi(M)}(\pi(T_C))$ of $\pi(T_C)$ in $\pi(M)$.
	Since $P$ is a multi-cell, $\pi(C)$ is a connected component of $\pi(P)(=\pi(M))$.
	We have $W \subseteq \pi(C)$ because $\mycl_{\pi(M)}(\pi(T_C)) \subseteq \pi(C)$.
	Let $x \in W$ and $y:=f_C(x) \in C$.
	Recall that every multi-cell is a submanifold.
	Choose a sufficiently small open box $B$ containing the point $y$, then we have $P \cap B=C \cap B$ and $C \cap B$ is the graph of the restriction of $f_C$ to $\pi(C) \cap \pi(B)$.
	$C \cap B$ is definable because $P \cap B$ is so.
	Pick a sufficiently small definable open subset $U$ of $\pi(C)$ so that $U \subseteq \pi(B) \cap W$. 
	Then, $\pi^{-1}(U) \cap C$ is definable and $U \cap \pi(T_C)$ is somewhere dense in $\pi(M)$.
	
	Observe that $g_C|_U$ is definable.
	Since  $U \cap \pi(T_C)$ is definable in $\mathfrak R$, $U \cap \pi(T_C)$ contains a nonempty definable open subset $V$ of $\pi(M)$ for the same reason as above.
	We may assume that $g_C|_V$ is continuous by \cite[8.3 Almost continuity]{Miller-dmin} by shrinking $V$ if necessary. 
	We can easily deduce $\pi(T_C) \cap V=\emptyset$ because $f_C < g_C$ on $V$.
	This is a contradiction.
\end{proof}

We are now ready to describe the statement of our key lemma.

\begin{lemma}[Key Lemma]\label{lem:main_lemma}
A  pseudo-special submanifold $M$ of $\mathbb R^n$ satisfying the following conditions is a multi-cell.
\begin{enumerate}
	\item[(1)] $M$ enjoys property $(\dagger)$.
	\item[(2)] $\pi(M)$ is a multi-cell.
	\item[(3)] $\mycl(M) \cap \pi^{-1}(x)=\mycl(M \cap \pi^{-1}(x))$ for every $x \in \pi(M)$;
	\item[(4)] There exists $N \in \mathbb N$ with $\myrank(\mycl(M \cap \pi^{-1}(x)))=N$ for every $x \in \pi(M)$;
	\item[(5)] If $N>1$, for every $1 \leq i <N$, the definable set $$P_i=\bigcup_{x \in \pi(M)} \myIso_i(\mycl(M) \cap \pi^{-1}(x))$$
	is a multi-cell.
	Observe that this condition implies that $(M,P_i)$ and $(P_j,P_k)$ are well-ordered pairs for $1 \leq i \leq N-1$ and $1 \leq j < k \leq N-1$;
	\item[(6)] $(M,P_i)$ and $(P_j,P_k)$ are extremely well-ordered pairs for $1 \leq i \leq N-1$ and $1 \leq j < k \leq N-1$.
\end{enumerate}
\end{lemma}

Our proof of the key lemma is long.
First we prove three technical lemmas necessary to prove the key lemma.
\begin{lemma}\label{lem:claim1}
	Let $M$ and $N$ be as in Lemma \ref{lem:main_lemma}.
	Let $X \in \mycon(M)$.
	Suppose $N=1$. We have $\myfr (X) \cap \pi^{-1}(\pi(M))=\emptyset$.
\end{lemma}
\begin{proof}
	We prove $\myfr X \cap \pi^{-1}(x)=\emptyset$ for every $x \in \pi(M)$.
	The closure of $M \cap \pi^{-1}(x)$ is discrete because $N=1$.
	This implies that $M \cap \pi^{-1}(x)$ is discrete and closed.
	We have $\myfr (X) \cap \pi^{-1}(x) \subseteq \myfr (M) \cap \pi^{-1}(x)=\myfr(M \cap \pi^{-1}(x))=\emptyset$ by clause (3) in Lemma \ref{lem:main_lemma}.
\end{proof}

\begin{lemma}\label{lem:claim2}
	Let $M$ and $N$ be as in Lemma \ref{lem:main_lemma}.
	Let $X \in \mycon(M)$.
	Suppose $N>1$. For every $b \in \myfr X \cap \pi^{-1}(\pi(M))$, there exist a definable open subset $V$ of $\pi(M)$ and connected graphs $C_1$ and $C_2$ on $V$ such that $b \in C_1$ and one of the following conditions is satisfied:
	\begin{itemize}
		\item $X>_{V}\underline{C_1}$, $\underline{C_2}>_{V}\underline{C_1}$ and every connected component of $X \cap (C_1,C_2)_{V}$ is a connected graph on $V$;
		\item $X<_{V}\underline{C_1}$, $\underline{C_2}<_{V}\underline{C_1}$  and every connected component of $X \cap (C_2,C_1)_{V}$ is a connected graph on $V$;
	\end{itemize}
	In particular, $\pi(b) \in \pi(X)$.
\end{lemma}
\begin{proof}
	Since $X$ is connected, $\pi(X)$ is contained in a connected component of $\pi(M)$, say $W$.
	Put $a:=\pi(b)$.
	Let $c$ be the last coordinate of $b$.
	
	Observe that $b \in \myfr X \cap \pi^{-1}(\pi(M))\subseteq \myfr M \cap \pi^{-1}(\pi(M))$ by using the assumption that $M$ is a submanifold of $\mathbb R^n$.
	By clauses (3-5) in Lemma \ref{lem:main_lemma}, 
	\begin{equation}
		\myfr M \cap \pi^{-1}(\pi(M))=P_1 \cup \ldots \cup P_{N-1} \tag{*} \label{eq:decomp}
	\end{equation}
	and $\pi(P_i)=\pi(M)$ for every $1 \leq i < N$.
	In particular, $W$ is a connected component of $\pi(P_i)$ for every $1 \leq i < N$.
	There exists $1 \leq l <N$ such that $b \in P_l$.
	Set $P:=P_l$ and $Q:=\myfr M \cap \pi^{-1}(\pi(M)) \setminus P$.
	
	Let $Z \in \mycon(P)$ containing $b$.
	$Z$ is a connected graph on $W$ by Lemma \ref{lem:easyeasy}, and we have either $X>_W\underline{Z}$ or $X<_W\underline{Z}$.
	We may assume $X>_W\underline{Z}$ without loss of generality.
	Let $S:=\{s \in \mathbb R\;|\; (a,s) \in Q, \ s>c\}$.
	We consider two separate cases.
	\medskip
	
	\textbf{Case 1.} Suppose $c \in \mycl S$.
	There exists a decreasing sequence $\{s_m\}_{m \in \mathbb N}$ in $S$ converging to $c$.
	Since $(a,s_m) \in Q$, there exists $1 \leq k <N$ such that $k \neq l$ and $P_k$ contains an infinite subset of $\{(a,s_m)\;|\; m \in \mathbb N\}$. 
	Taking a subsequence, we may assume that $(a,s_m) \in P_k$ for all $m \in \mathbb N$.
	Put $R:=P_k$.
	Let $Y_m \in \mycon(R)$ containing the point $(a,s_m)$.
	Observe that $Y_m$ is a connected graph on $W$ by Lemma \ref{lem:easyeasy}.
	We have $f_{Y_m}(x)>(c+s_m)/2$ for every point $x \in W$ sufficiently close to $a$ because $f_{Y_m}$ is continuous.
	Since $b \in \myfr X \cap \pi^{-1}(W)$, we can choose $x \in W$ and $t \in \mathbb R$ such that $(x,t) \in X$, and $x$ and $t$ are sufficiently close to $a$ and $c$, respectively.
	This implies that $(x,t) \in Y_m^-$.
	We obtain $X <_W \underline{Y_m}$.

	Put $Y':=\myfr(\bigcup_{m \in \mathbb N} Y_m) \cap \pi^{-1}(W)$.
	Observe that $b \in Y'$.
	We show $Y' \subseteq Z$.
	Assume for contradiction that there exists $b' \in Y' \setminus Z$.
	Since $b' \in \myfr M \cap \pi^{-1}(\pi(M))$, there exist $1 \leq l' <N$ and $R' \in \mycon(P_{l'})$ with $b' \in R'$ by equality (\ref{eq:decomp}).
	By clause (5) in Lemma \ref{lem:main_lemma} and Lemma \ref{lem:easyeasy}, $R'$ is a connected graph on $W$.
	In particular, there exists $c' \in \mathbb R$ with $(a,c') \in R'$.
	Since $X <_W \underline{Y_m}$ and $\underline{Z} <_{W}X$, we have $\underline{Z} <_W \underline{Y_m}$ for each $m \in \mathbb N$.
	Since $Z \neq R'$, the definition of $b'$ and the fact $b' \in R'$ imply $\underline{R'}>_W\underline{Z}$.
	We have $c'>c$ and $c'<s_m$, which contradicts to $c=\lim_{m \to \infty}s_m$.
	We have shown $Y' \subseteq Z$.
	
	Pick a point $(x,s) \in X$.
	Since $X >_W \underline{Z}$, there exists $s>u$ with $(x,u) \in Z$.
	Since $X <_W \underline{Y_m}$, for every $m \in \mathbb N$, there exists $u_m>s$ such that $(x,u_m) \in Y_m$.
	We have $(x, \inf\{u_m\;|\; m \in \mathbb N\}) \in Y' \setminus Z$, which is absurd.
	\medskip
	
	\textbf{Case 2.} Next suppose $c \notin \mycl(S)$.
	We want to show $b \notin \bigcup_{i=1}^{N-1} \mycl(P_i \cap Z^+).$
	Since $c \notin \mycl(S)$, $c \notin \myadh^+(P_i,P,a)$ for all $1 \leq i <l$.
	We have $b \notin \myadh^+(P_i,P)$ because the pair $(P_i,P)$ is an extremely well-ordered pair by clause (6) of Lemma \ref{lem:main_lemma}.
	This implies that $b \notin \myfr(P_i \cap Z^+)$ for $1 \leq i <l$ because $b \in Z$.
	
	Put $c':=\inf\{z >c\;|\; (a,z) \in \bigcup_{i \geq l}P_i\}$.
	If $c'=\infty$, we have $P_i \cap Z^+=\emptyset$ for $i \geq l$.
	If $c'<\infty$, there exist $l \leq l' <N$ and a connected component $Z'$ of $P_{l'}$ such that $(a,c') \in Z'$.
	We have $\pi(Z')=W$ because $P_{l'}$ is a multi-cell.
	For every $l \leq i <N$, every connected component $D$ of $P_i$ with $\pi(D)=W$ is a connected graph by Lemma \ref{lem:easyeasy} and clause (5) of Lemma \ref{lem:main_lemma}.
	We have either $\underline{D}<_W \underline{Z}$ or $\underline{D} >_W \underline{Z}$ if $D \neq Z$.
	We have $\underline{D}<_W \underline{Z'}$ or $\underline{D} >_W \underline{Z'}$ if $D \neq Z'$.
	Therefore, $D \cap (Z,Z')_W=\emptyset$ by the definition of $c'$.
	This implies that $(Z,Z')_W \cap P_i=\emptyset$ for every $i \geq l$.
	Consequently, we have $b \notin \bigcup_{i=1}^{N-1} \mycl(P_i \cap Z^+)$.
	
	We have $\pi^{-1}(W) \cap \myfr(M) \cap Z^+ = \bigcup_{i=1}^{N-1} \pi^{-1}(W) \cap P_i \cap Z^+ \subseteq \bigcup_{i=1}^{N-1} \mycl(P_i \cap Z^+)$.
	Therefore, there exists a bounded open box $B$ containing $b$ with 
	\begin{center}
		$\pi(B) \cap \pi(M) \subseteq W$ and $B \cap \myfr(M) \cap Z^+=\emptyset$
	\end{center}
	because $b \notin \bigcup_{i=1}^{N-1} \mycl(P_i \cap Z^+)$.
	Let $I$ be an open interval with $B=\pi(B) \times I$.
	
	Since $X>_W\underline{Z}$ and $b \in \myfr(X) \cap Z$, we have $b \in \myfr(M \cap Z^+)$.
	This means $b \in \myadh^+(M,P)$.
	We have $b \in \myadh^+(M,P,a)$ because $(M,P)$ is an extremely well-ordered pair by clause (6).
	This implies that the closure of $D:=\{t>c\;|\; (a,t) \in M \cap B\}$ contains $c$.
	Clauses (3,4) imply that $D$ has an empty interior. 
	Take $c_1 \in D$.
	Put $c_m:=\sup\{t \in D\;|\; t<c_{m-1}\}$ for $m>1$.
	Since $B \cap \myfr(M) \cap Z^+=\emptyset$, we have $(a,c_m) \in M$.
	Observe that $M \cap (\{a\} \times (c_m,c_{m-1}))=\emptyset$ for $m>1$.
	We also have $\lim_{m \to \infty}c_m=c$.
	If not, $\lim_{m \to \infty}c_m \in M$ for the same reason as above.
	On the other hand,  because $M$ is pseudo-special, there exists an open interval $J$ containing $\lim_{m \to \infty}c_m$ with $(\{a\} \times J) \cap M=\{(a,\lim_{m \to \infty}c_m)\}$, which is absurd.
	
	Let $D'_1 \in \mycon(M)$ with $(a,c_1) \in D'_1$.
	Since $M$ is pseudo-special, there exists an open box $B'$ in $\mathbb R^n$ such that $(a,c_1) \in B'$ and $B' \cap D'_1$ is the graph of a continuous function defined on $\pi(B') \cap \pi(M)$.
	Shrinking $B'$ if necessary, we may assume that $\pi(B') \cap \pi(M)=\pi(B') \cap \pi(D'_1)$.
	Put 
	\begin{center}
		$V:=\pi(B') \cap \pi(D'_1)$ and $C_1:=Z \cap \pi^{-1}(V)$.
	\end{center}
	
	Let $D_m \in \mycon(M \cap (V \times I))$ containing $(a,c_m)$ for $m \in \mathbb N$.
	Put 
	\begin{center}
		$C_2:=D_1$.
	\end{center}
	We prove $$\pi(D_m)=V$$ by induction on $m$.
	Observe that $D_1$ is the graph of a continuous function on $V$.
	Therefore, we have $\pi(D_1)=V$.
	Suppose $m>1$.
	By the induction hypothesis, $D_{m-1}$ is the graph of a continuous function $g_{m-1}$ defined on $V$.
	Since $c_m<c_{m-1}$, $D_m$ is connected and $D_m \cap D_{m-1}=\emptyset$, we have $D_m<_{V}\underline{D_{m-1}}$.
	Assume for contradiction that $\pi(D_m) \neq V$.
	Observe that $$V_m:=\pi(D_m)$$ is open in $V$ because $M$ is pseudo-special.
	Since $\underline{C_1} <_{V} D_m<_{V}\underline{D_{m-1}}$, by Lemma \ref{lem:easy}, $\myfr(V_m) \cap  V \neq \emptyset$ and, for every $x' \in \myfr(V_m) \cap  V$, there exists $y' \in \myfr(D_m) \cap [C_1,D_{m-1}]_{V}$ with $\pi(y')=x'$.
	
	Pick $x' \in \myfr(V_m) \cap  V$.
	Let $c'$ be the last coordinate of $y'$.
	We have $f_{C_1}(x') \leq c' \leq g_{m-1}(x')$.
	In particular, $y'=(x',c') \in V \times I$.
	We have $y' \in \myfr(M)$ because $D_m \in \mycon(M \cap (V \times I))$ and $M$ is pseudo-special.
	We have $c'=f_{C_1}(x')$ because $V \times I \subseteq B$ and $B \cap \myfr(M) \cap Z^+=\emptyset$.
	In other words, $y' \in C_1$.
	Observe that $y' \in \mycl(M \cap Z^+)$ because $y' \in \myfr(D_m)$ and $\underline{C_1} <_{V} D_m$.
	These imply $y' \in \myadh^+(M,P)$.
	By clause (5), we have $y' \in \myadh^+(M,P,x')$.
	This deduces that there exists a point $d' \in \mathbb R$ with $(x',d') \in M$ and $c'<d'<g_{m-1}(x')$.
	Let $C \in \mycon(M \cap (V \times I))$ with $(x',d') \in C$.
	Since $M$ is pseudo-special, if we choose a sufficiently small open box $B''$ containing $(x',d')$, $M \cap B''$ is the graph of a continuous function defined on the submanifold $\pi(M) \cap \pi(B'')$.
	We may assume that $\pi(M) \cap \pi(B'')$ is connected by shrinking $B''$ if necessary.
	Under this condition, we have $M \cap B''=C \cap B''$ and $\pi(M) \cap \pi(B'') = \pi(C) \cap \pi(B'')$.
	Since $\pi(C) \cap \pi(B'')$ is a neighborhood of $x'$ in $\pi(M)$, $C \cap \pi^{-1}(V_m)$ is not empty because $x' \in \myfr(V_m)$.
	Pick $C' \in \mycon(C \cap \pi^{-1}(V_m))$.
	$\pi(C')$ is open in $V_m$ because $M$ is pseudo-special.
	We have $\underline{D_m}<_{V'_m}C'<_{V_m}\underline{(D_{m-1}\cap \pi^{-1}(V_m))}$.
	We also have $\pi(C') \neq V_m$ because $M \cap (\{a\} \times (c_m,c_{m-1}))=\emptyset$.
	By Lemma \ref{lem:easy}, $\myfr(C') \cap [D_m,D_{m-1}\cap \pi^{-1}(V_m)]_{V_m} \neq \emptyset$. 
	Pick $z \in \myfr(C') \cap [D_m,D_{m-1}\cap \pi^{-1}(V_m)]_{V_m}$.
	Since $C'$ is a connected component of $M \cap (V_m \times I)$ and $M$ is pseudo-special, $z \in \myfr(M)$.
	This implies that $B \cap \myfr(M) \cap Z^+ \neq \emptyset$, which is absurd.
	We have shown $\pi(D_m)=V$ for all $m \in \mathbb N$.
	
	Next we show $$M \cap (C_1,C_2)_{V}= \bigcup_{m>1}D_m.$$
	Let $z \in M \cap (C_1,C_2)_{V} \setminus(\bigcup_{m>1}D_m)$ for contradiction.
	Put $x_z:=\pi(z)$ and $c_z$ be the last coordinate of $z$.
	For $m \in \mathbb N$, there exists $d_m \in \mathbb R$ such that $(x_z,d_m) \in D_m$ because $\pi(D_m)=V$.
	We have $\inf \{d_m\;|\;m \in \mathbb N\}=f_{C_1}(x')$ because $B \cap \myfr(M) \cap Z^+=\emptyset$ and $M$ is pseudo-special.
	We can choose $m >1$ such that $d_m < c_z < d_1$.
	Let $C'' \in \mycon(M \cap (C_1,C_2)_{V})$ with $z \in C''$.
	We have $\pi(C'') \neq V$ because $M \cap (\{a\} \times (c,c_1))=\{c_m\;|\;m>1\}$.
	By Lemma \ref{lem:easy}, we have $\myfr(C'') \cap [D_m,C_2]_{V} \neq \emptyset$, which contradicts $B \cap \myfr(M) \cap Z^+=\emptyset$.
	
	Since $M \cap (C_1,C_2)_{V}= \bigcup_{m>1}D_m$, every connected component of $X \cap (C_1,C_2)_{V}$ coincides with one of $D_m$'s.
	Therefore, every connected component of $X \cap (C_1,C_2)_{V}$ is a connected graph on $V$.
	The `in particular' part of this lemma is obvious.
\end{proof}

\begin{lemma}\label{lem:claim3}
	Let $M$ and $N$ be as in Lemma \ref{lem:main_lemma}.
	Let $X \in \mycon(M)$.
	Let $V$ be a connected component of $\pi(M)$ containing $\pi(X)$.
	Let $b \in X$ and put $a=\pi(b)$.
	\begin{enumerate}
		\item[(a)] For every continuous map $\alpha:[0,1] \to V$ with $\alpha(0)=a$, there exists a unique continuous map $\beta:[0,1] \to X$ such that $\alpha=\pi \circ \beta$ and $\beta(0)=b$.
		\item[(b)] Let $\beta_i:[0,1] \to X$ be continuous maps with $\beta_i(0)= b$ for $i=0,1$ and let $H:[0,1] \times [0,1] \to V$ be a homotopy with $H(s,i)=\pi(\beta_i(s))$ for $i=0,1$ and $s \in [0,1]$, and $H(0,t)=a$ for $t \in [0,1]$. 
		Then, there exists a unique homotopy $\widetilde{H}:[0,1] \times [0,1] \to X$ between $\beta_0$ and $\beta_1$ with $\widetilde{H}(0,t)=b$ for $t \in [0,1]$. 
	\end{enumerate}
\end{lemma}
\begin{proof}
	(a) 
	First we show uniqueness.
	Let $\beta_1,\beta_2:[0,1] \to X$ be continuous maps with $\alpha=\pi \circ \beta_i$ and $\beta_i(0)=b$ for $i=1,2$.
	Let $I:=\{t \in [0,1]\;|\; \beta_1(t)=\beta_2(t)\}$, which is closed in $[0,1]$.
	We show that $I$ is open in $[0,1]$.
	Let $t_0 \in I$ be an arbitrary point.
	Put $y_0:=\beta_1(t_0)=\beta_2(t_0)$.
	Since $M$ is a pseudo-special submanifold, there exists an open box $B$ containing $y_0$ such that $M \cap B$ is the graph of a continuous function $f$ defined on $\pi(B) \cap \pi(M)$.
	We have $\beta_1(t)=(\alpha(t),f (\alpha(t)))=\beta_2(t)$ for every $t \in [0,1]$ sufficiently close to $t_0$.
	This means that $I$ is open in $[0,1]$.
	We have $I=[0,1]$ because $[0,1]$ is connected.
	
	Next we show existence.
	Consider the set $J$ of $t \in [0,1]$ such that there exists a continuous map $\beta_t:[0,t] \to X$ with $\beta_t(0)=a$ and $\pi \circ \beta_t=\alpha|_{[0,t]}$.
	$J$ is open because $M$ is pseudo-special.
	Let $t_0:=\sup J$.
	For every $0 \leq t <t_0$, there exists a unique continuous map $\beta_t:[0,t] \to X$ satisfying the above conditions.
	The map $\beta:[0,t_0) \to X$ defined by $\beta(t):=\beta_t(t)$ is a continuous map by uniqueness of $\beta_t$.
	Let $(s_m)_{m \in \mathbb N}$ be an increasing sequence in $[0,t_0)$ converging to $t_0$.
	By clause (1) of Lemma \ref{lem:main_lemma}, $M$ enjoys property $(\dagger)$.
	By property $(\dagger)$, we may assume that $(\beta(s_m))_{m \in \mathbb N}$ is a convergent sequence by taking a subsequence. 
	Put $c=\lim_{m \to \infty}\beta(s_m)$.
	We have $\pi(c)=\alpha(t_0) \in \pi(M)$.
	Therefore, we obtain $c \in \mycl(X) \cap \pi^{-1}(\pi(M))$.
	
	Assume for contradiction that $c \in \myfr(X)$.
	If $N=1$, this contradicts Lemma \ref{lem:claim1}.
	Suppose $N>1$.
	By Lemma \ref{lem:claim2}, there exist a definable open subset $V'$ of $\pi(M)$ and connected graphs $C_1$ and $C_2$ on $V'$ such that $c \in C_1$ and one of the following conditions is satisfied:
	\begin{enumerate}
		\item[(i)] $X>_{V'}\underline{C_1}$, $\underline{C_2}>_{V'}\underline{C_1}$ and every connected component of $X \cap (C_1,C_2)_{V'}$ is a connected graph on $V'$;
		\item[(ii)] $X<_{V'}\underline{Z'}$, $\underline{C_2}<_{V'}\underline{C_1}$ and every connected component of $X \cap (C_2,C_1)_{V'}$ is a connected graph on $V'$;
	\end{enumerate}
	We may assume that condition (i) is satisfied without loss of generality.
	By taking $u_0<t_0$ sufficiently close to $t_0$, we may assume that $\alpha([u_0,t_0])$ is contained in $V'$.
	By taking a subsequence of $(s_m)_{m \in \mathbb N}$, we may assume that $\beta(s_m) \in X \cap (C_1,C_2)_{V'}$ and $u_0<s_m<t_0$ for every $m \in \mathbb N$.
	By using the assumption that $M$ is a pseudo-special submanifold and condition (i), we can easily prove that $K:=\{t \in [s_1,t_0)\;|\;\beta(t) \in X \cap (C_1,C_2)_{V'}\}$ is closed and open in $[s_1,t_0)$.
	We omit the details.
	Since $s_1 \in K$, we have $K=[s_1,t_0)$.
	This means that $\beta([s_1,t_0)) \subseteq X \cap (C_1,C_2)_{V'}$.
	Since $\beta([s_1,t_0))$ is connected, it is contained in a connected component of $X \cap (C_1,C_2)_{V'}$, say $D$.
	However, $D$ is a connected graph on $V'$ with $C_1<_{V'}D$ by condition (i).
	We have $c \in \mycl(D)$, but $c \in C_1$ and $C_1 \cap \mycl(D)=\emptyset$, which is absurd. 
	
	We have $c \in X$.
	This implies $t_0 \in J$.
	Because $J$ is open in $[0,1]$, we get $t_0=1$.
	
	(b) Uniqueness follows from clause (a).
	
	By clause (a), for every $0 \leq s \leq 1$, there exists a unique continuous map $\beta[s]:[0,1] \to V$ such that $\beta[s](0)=b$ and $H(s,t)=\pi(\beta[s](t))$ for $0 \leq t \leq 1$.
	Put $\widetilde{H}(s,t):=\beta[s](t)$.
	It is obvious that $\widetilde{H}$ satisfies the conditions in clause (b) except continuity.
	The remaining task is to show that $\widetilde{H}$ is continuous.
	Let $$K:=\{s \in [0,1]\;|\; \widetilde{H}|_{[0,s] \times [0,1]} \text{ is continuous}\}.$$
	Assume for contradiction that $s_0:=\sup K < 1$.
	Since $M$ is pseudo-special and uniqueness of $\widetilde{H}$, for every $0 \leq t \leq 1$, there exist $\delta_t>s_0$ and $0 \leq \varepsilon_{1,t} <t < \varepsilon_{2,t} \leq 1$ such that the restriction of $\widetilde{H}$  to $[0,\delta_t] \times [\varepsilon_{1,t},\varepsilon_{2,t}]$ is continuous.
	Since $[0,1]$ is compact, there exists $\delta>s_0$ such that the restriction of $\widetilde{H}$  to $[0,\delta] \times [0,1]$ is continuous, which is absurd.
\end{proof}

Now, we are ready to prove Lemma \ref{lem:main_lemma}.

\begin{proof}[Proof of Lemma \ref{lem:main_lemma}]
	Let $X \in \mycon(M)$.
	Since $X$ is connected, $\pi(X)$ is contained in a connected component of $\pi(M)$, say $V$.
	Observe that $V$ is simply connected by clause (2) and the definition of multi-cells.
	We have only to show that $\pi|_X$ is a homeomorphism onto $V$ by clause (2).
	Observe that $\pi|_X$ is a local homeomorphism because $\mathcal M$ is pseudo-special.

	 We show $\pi(X)=V$.
 	Since $\pi|_X$ is a local homeomorphism, $\pi(X)$ is open in $V$.
 	Assume for contradiction that $\pi(X)$ is not closed in $V$.
 	We can find $a \in \myfr(\pi(X)) \cap V$.
 	Let $(a_m)_{m \in \mathbb N}$ be a convergent sequence in $\pi(X)$ converging to $a$.
 	Choose $b_m \in X$ with $\pi(b_m)=a_m$.
 	Since $M$ enjoys property $(\dagger)$, we may assume that $(b_m)_{m \in \mathbb N}$ is convergent in $\mathbb R^n$ by taking a subsequence.
 	Let $b:=\lim_{m \to \infty} b_m$.
 	We have $\pi(b)=a \notin \pi(X)$ and $b \notin X$.
 	This implies that $b \in \myfr X \cap \pi^{-1}(\pi(M))$ and $\pi(b) \notin \pi(X)$.
 	These contradict Lemma \ref{lem:claim1} if $N=1$ and `in particular' part of Lemma \ref{lem:claim2} if $N>1$.
	 
	 The remaining task is to show $\pi|_X$ is injective.
	 Let $y_1,y_2 \in X$ with $x_0:=\pi(y_1)=\pi(y_2)$.
	 Observe that $X$ is a connected submanifold of $\mathbb R^n$; therefore, $X$ is pathwise connected.
	 Let $\gamma:[0,1] \to X$ be a continuous map connecting $y_1$ and $y_2$.
	 Put $\beta=\pi \circ \gamma$.
	 If $\gamma$ is a constant map, we have $y_1=y_2$.
	 Therefore, we may assume that $\gamma$ is not a constant map.
	 Observe that $\beta$ is not a constant map because $M$ is pseudo-special.
	 Since $V$ is simply connected, there exists a homotopy $H: [0,1] \times [0,1] \to V$ such that $H(0,t)=\beta(t)$ and $H(1,t)=x_0$.
	 By Lemma \ref{lem:claim3}, there exists a homotopy $\widetilde{H}: [0,1] \times [0,1] \to X$ such that $\widetilde{H}(0,t)=\gamma(t)$ and $\widetilde{H}(1,t)=y_1$.
	 We have $y_2=\gamma(1)=\widetilde{H}(1,1)=y_1$.
\end{proof}

Finally, we prove Theorem \ref{thm:decomposition_into_multicells}.

\begin{proof}[Proof of Theorem \ref{thm:decomposition_into_multicells}]
We first reduce to the case where $k=1$.
Suppose that Theorem \ref{thm:decomposition_into_multicells} holds for $k=1$.
Put $A_i(0):=A_i$ and $A_i(1)= \mathbb R^n \setminus A_i$ for $ 0 \leq i \leq k$.
Let $\mathcal S$ be the set of sequences $\overline{s}$ of zeros and ones of length $k$ with $\overline{s} \neq (1,\ldots,1)$.
Set $A_{\overline{s}}:= \bigcap_{i=1}^k A_i(s_i)$ for $\overline{s}=(s_1,\ldots, s_k) \in \mathcal S$.
Since we assume that Theorem \ref{thm:decomposition_into_multicells} holds for $k=1$, for every $\overline{s} \in \mathcal S$, $A_{\overline{s}}$ is partitioned into finitely many multi-cells, say, $C_{\overline{s},1}, \ldots, C_{\overline{s},l(\overline{s})}$.
If $n=1$, $\{C_{\overline{s},i}\;|\; 1 \leq i \leq l(\overline{s}), \overline{s} \in \mathcal S\}$ is a desired multi-cell decomposition.
If $n>1$, there exists a multi-cell decomposition partitioning $\{\pi(C_{\overline{s},i})\;|\; 1 \leq i \leq l(\overline{s}), \overline{s} \in \mathcal S\}$, say $D_1,\ldots, D_m$.
Let $\mathcal D$ be the collection of nonempty definable sets of the form $C_{\overline{s},i} \cap D_j$ for some $\overline{s} \in \mathcal S$, $1 \leq i \leq l(\overline{s})$ and $1 \leq j \leq m$.
$\mathcal D$ is a desired multi-cell decomposition.
Therefore, we may assume that $k=1$.
We denote $A_1$ by $M$ for simplicity.
We use this reduction process in this proof without notice.

We show this theorem by induction on $n$.
Suppose $n=1$.
Let $J=\myint(M)$. 
Let $N=\myrank(\mycl(M \setminus J))$ and set $S_i:=\myIso_i(\mycl(M \setminus J)) \cap (M \setminus J)$ for $0 \leq i \leq N-1$.
It is easy to show that $\{J, S_0,\ldots, S_{N-1}\}$ is a multi-cell decomposition of $M$ after removing empty sets in the collection.
We omit the details.

Suppose $n>1$.
By the induction hypothesis, we can easily reduce to the case where $\pi(M)$ is a multi-cell.
This reduction process is also used without notice.

First we consider a special case.
Suppose that $\dim M \cap \pi^{-1}(x) =0$ for every $x \in \pi(M)$.
We reduce to simpler cases in a step-by-step manner.
Set $$N:=\max\{\myrank(\mycl(M \cap \pi^{-1}(x)))\;|\; x \in \pi(M)\}.$$
Set $A':=\bigcup_{x \in \pi(M)} \{x\} \times \myIso( \mycl(M \cap \pi^{-1}(x)))$.
By induction on $N$, by considering $A'$ instead of $M$, we may assume that $M \cap \pi^{-1}(x)$ is discrete for every $x \in \pi(M)$.

Put $M[i]:=M \cap \pi^{-1}(\{x \in \pi(M)\;|\;\myrank(\mycl(M \cap \pi^{-1}(x)))=i\})$ for $1 \leq i \leq N$.
We may assume that $\myrank(\mycl(M \cap \pi^{-1}(x)))$ is independent of $x \in \pi(M)$, say $N$, by considering $M[i]$ instead of $M$.
We define $P_k$ for $1 \leq k \leq N-1$ as in clause (5) of Lemma \ref{lem:main_lemma}.
Put $P_0:=M$ for simplicity of notation.
We prove the following claim which implies a multi-cell decomposition partitioning $M$.
\medskip

\textbf{Claim 1.} There exists a multi-cell decomposition $\{C_1, \ldots, C_l\}$ partitioning $\pi(M)$ such that, for $1 \leq i \leq l$,  $P_i \cap \pi^{-1}(C_j)$ is a multi-cell for $0 \leq i \leq N-1$ and clause (6) in Lemma \ref{lem:main_lemma} is satisfied for $M \cap \pi^{-1}(C_i)$.

\begin{proof}[Proof of Claim 1]
	We prove this claim by induction on $N$.
	First we reduce to a simpler case.
	There exists a multi-cell decomposition partitioning $\pi(M)$, say $D_1, \ldots, D_m$  by induction hypothesis.
	By considering $M \cap \pi^{-1}(D_i)$ instead of $M$, we may assume that $\pi(M)$ is a multi-cell in $\mathbb R^{n-1}$ of dimension $d$.
	Therefore, clause (2) of Lemma \ref{lem:main_lemma} is always satisfied.
	
	Let $\myreg(M)$ be the set of points $x \in M$ such that there exists an open box $B$ in $\mathbb R^n$ containing $x$ such that $\pi|_{B \cap M}$ is a homeomorphism onto $\pi(B) \cap \pi(M)$.
	Assume for contradiction, there exists a nonempty definable open subset $U$ of $\pi(M)$ contained in $\pi(M \setminus \myreg(M))$.
	Since $M \cap \pi^{-1}(x)$ is discrete for $x \in U$, by \cite{Miller-choice}, there exist definable functions $f,g_1,g_2:U \to \mathbb R$ such that, for every $x \in U$, $(x,f(x)) \in M \setminus \myreg(M)$ and $M \cap (\{x\} \times (g_1(x),g_2(x))) =\{(x,f(x))\}$.
	By \cite[8.3 Almost continuity]{Miller-dmin}, we may assume that $f,g_1,g_2$ are continuous by shrinking $U$ if necessary.
	This contradicts that $f(x) \notin \myreg(M)$.
	Put $S_1:=\pi(M \setminus \myreg(M))$, which is nowhere dense in $\pi(M)$.
	
	Put $S_2:=\{x \in \pi(M)\;|\; \mycl(M) \cap \pi^{-1}(x) \neq \mycl(M \cap \pi^{-1}(x))\}$, which is nowhere dense in $\pi(M)$ by \cite[Main Lemma(5)]{Miller-dmin}.
	By induction on $d$, we may assume that $M$ is pseudo-special and clauses (3,4) in Lemma \ref{lem:main_lemma} are satisfied by considering $M \setminus \pi^{-1}(\mycl(S_1 \cup S_2))$ instead of $M$. 
	Observe that clause (1) of Lemma \ref{lem:main_lemma} always holds in both cases (a) and (b) of this theorem.
		
	First suppose $N=1$.
	We have $\myrank(M \cap \pi^{-1}(x))=1$ for each $x \in \pi(M)$.
	By the induction hypothesis, there exists a multi-cell decomposition $\{D_1, \ldots, D_m\}$ of $\pi(M)$.
	We may assume that $\pi(M)$ is a multi-cell by considering $M \cap \pi^{-1}(D_i)$ instead of $M$.
	Observe that clauses (5,6) in Lemma \ref{lem:main_lemma} impose no conditions in this case.
	$M$ is a multi-cell by Lemma \ref{lem:main_lemma}.
		
	Suppose $N>1$.
	Apply the induction hypothesis to $P_1$, then we may assume that $P_1, \ldots, P_{N-1}$ are multi-cells and $(P_i,P_j)$ are extremely well-ordered pairs for $1 \leq i < j \leq N-1$.
	Observe that $\pi(M)$ is a multi-cell because $\pi(M)=\pi(P_1)$ and $P_1$ is a multi-cell.
	We want to reduce to the case where $(M,P_i)$ is an extremely well-ordered pair for every $1 \leq i \leq N-1$.
	Put $T_{1i}:=\{x \in \pi(M)\;|\; \myadh^+(M,P_i) \cap \pi^{-1}(x) \neq \myadh^+(M,P_i,x)\}$ and $T_{2i}:=\{x \in \pi(M)\;|\; \myadh^-(M,P_i) \cap \pi^{-1}(x) \neq \myadh^-(M,P_i,x)\}$.
	Put $T:=\bigcup_{i=1}^{N-1} T_{1i} \cup T_{2i}$.
	$T$ is nowhere dense by Lemma \ref{lem:well_ordered}.
	
	We treat this case by induction on $d$.
	If $d=0$, then $T=\emptyset$.
	$M$ is a multi-cell by Lemma \ref{lem:main_lemma}.
	If $d>0$, the claim holds for $M \cap \pi^{-1}(T)$ because $\dim T<d$.
	Let $\{D_1,\ldots, D_l\}$ be a multi-cell decomposition of $\pi(M) \setminus T$.
	By considering $M \cap \pi^{-1}(D_i)$ instead of $M$, we may assume that $T=\emptyset$ and $\pi(M)$ is a multi-cell.
	We have succeeded in reducing to the case where $M$ is a pseudo-special submanifold enjoying clauses (1) through (6) in Lemma \ref{lem:main_lemma}.
	Lemma \ref{lem:main_lemma} implies that $M$ is a multi-cell.
\end{proof}
We have proved Theorem \ref{thm:decomposition_into_multicells} for the case in which $\dim M \cap \pi^{-1}(x) =0$ for every $x \in \pi(M)$.
We treat the general case.

For every set $S \subseteq \mathbb R^n$, put $S_x:=\{y \in \mathbb R\;|\; (x,y) \in S\}.$
Consider the set $J=\bigcup_{x \in \pi(M)} \{x\} \times \myint(M_x)$.
By Claim 1, there exists a multi-cell decomposition partitioning $M \setminus J$.
By considering $J$ instead of $M$, we may assume that $M_x$ is open for $x \in \pi(M)$.
Consider the following sets:
\begin{align*}
	&V_{+\infty}:=\{(x,y) \in M\;|\;\ \forall t \ (t>y \rightarrow (x,t) \in M) \};\\
	&V_{-\infty}:=\{(x,y) \in M\;|\;\ \forall t \ (t<y \rightarrow (x,t) \in M) \};\\
	&V_{\mathbb R}:=V_{+\infty} \cap V_{-\infty};\\
	&W:=M \setminus (V_{+\infty} \cup V_{-\infty}).
\end{align*}
A multi-cell decomposition partitioning $\{V_{+\infty},V_{+\infty},V_{\mathbb R},W\}$ is a multi-cell decomposition partitioning $M$.
Therefore, we have only to construct multi-cell decompositions partitioning each of $V_{+\infty}$, $V_{+\infty}$, $V_{\mathbb R}$ and $W$.
Observe that $V_{\mathbb R}=\pi(V_{\mathbb R}) \times \mathbb R$.
It is easy to show that $V_{\mathbb R}$ is a union of multi-cells.
We omit the details.
We next show that there exists a multi-cell decomposition partitioning $V_{+\infty}$.
By considering a multi-cell decomposition of $\pi(V_{+\infty})$, we may reduce to the case where $\pi(V_{+\infty})$ is a multi-cell.
Consider the definable function $f:\pi(V_{+\infty}) \to \mathbb R$ given by $ f(x)=\inf\{t \in \mathbb R\;|\; (x,t) \in V_{+\infty}\}$.
Let $D$ be the set of points at which $f$ is discontinuous. 
By \cite[8.3 Almost continuity]{Miller-dmin}, we have $\dim D < \dim \pi(V_{+\infty})$.
We can reduce to the case in which $f$ is continuous by induction on $\dim \pi(V_{+\infty})$.
$V_{+\infty}$ is a multi-cell because $\pi(V_{+\infty})$ is a multi-cell and $V_{+\infty}=\{(x,t) \in \pi(V_{+\infty})\times \mathbb R\;|\; t>f(x)\}$.
We can prove that there exists a multi-cell decomposition partitioning $V_{-\infty}$ in the same manner.

The remaining task is to construct a multi-cell decomposition partitioning $W$.
Observe that every connected component of $W_x$ is a bounded open interval for every $x \in \pi(W)$.
As usual, we can reduce to the case in which $\pi(W)$ is a multi-cell by the induction hypothesis. 
Consider the definable set
\begin{align*}
&Z:=\{(x,y_1,y_2) \in \pi(W) \times \mathbb R \times \mathbb R\;|\; (x,y_1) \notin W, (x,y_2) \notin W, \\
&\qquad y_1<y_2, (\forall t \ y_1<t<y_2 \rightarrow (x,t) \in W)\}.
\end{align*}
Let $Q:=\{(x,y) \in \pi(W) \times \mathbb R\;|\; \exists z\ (x,y,z) \in Z\}$.
Observe that $\pi(W)=\pi(Q)$ and $\dim Q \cap \pi^{-1}(x)=0$ for every $x \in \pi(Q)$.
By Claim 1, there exists a multi-cell decomposition $\{Q_1,\ldots, Q_l\}$ of $Q$.
Observe that, for every $(x,y) \in Q$, there exists a unique $z \in \mathbb R$ with $(x,y,z) \in Z$.
Let $g:Q \to \mathbb R$ be the definable function sending $(x,y) \in Q$ to the unique point $z \in \mathbb R$ with $(x,y,z) \in Z$.
Put $g_i:=g|_{Q_i}$ for $1 \leq i \leq l$.
By \cite[8.3 Almost continuity]{Miller-dmin}, the set $G_i$ of points at which $g_i$ is discontinuous is nowhere dense in $Q_i$.
For every $C \in \mycon(Q_i)$, $C$ is the graph of a continuous function defined on a connected component of $\pi(Q_i)$.
Therefore, $\pi(G_i \cap C)$ is nowhere dense in $\pi(Q_i)$ for every $C \in \mycon(Q_i)$. 
Since $\mycon(Q_i)$ is a countable set, $\pi(G_i)=\bigcup_{C \in \mycon(Q_i)}\pi(G_i \cap C)$ has an empty interior in $\pi(Q_i)$ by the Baire category theorem.
Since $\pi(G_i)$ is definable in the d-minimal structure $\mathfrak R$, $\dim \pi(G_i)<\dim \pi(Q_i)$.
By induction on $\dim(\pi(Q_i))$, using the induction hypothesis, there exists a multi-cell decomposition $D_1, \ldots, D_m$ of $\pi(Q)$ such that $Q \cap \pi^{-1}(D_i)$ is a multi-cell and the restriction of $g$ to $Q \cap \pi^{-1}(D_i)$ is continuous.
Since $Q \cap \pi^{-1}(D_i)$ is a multi-cell, every connected component of $W \cap \pi^{-1}(D_i)$ is of the form $\{(x,y) \in D_i \times \mathbb R\;|\; f(x)<y<g(x,f(x))\}$, where $f$ is a continuous function defined on $D_i$.
This means that $W \cap \pi^{-1}(D_i)$ is a multi-cell.
\end{proof}

\section{A generalized version of Theorem \ref{thm:connected_expansion}}\label{sec:construction}

Theorem \ref{thm:connected_expansion} immediately follows from the following theorem:
\begin{theorem}\label{thm:natural}
	Let $\mathfrak R$ be a d-minimal expansion $(\mathbb R,<,+,\ldots)$ of the ordered group of reals which satisfies one of the following conditions:
	\begin{enumerate}
		\item[(a)] There exists a definable homeomorphism between $\mathbb R$ and a bounded open interval;
		\item[(b)] For every $n \in \mathbb N$ and every pseudo-special submanifold $M$ of $\mathbb R^n$, property $(\dagger)$ is fulfilled.
	\end{enumerate}
	A subset $X$ of $\mathbb R^n$ is $\mathfrak R^\natural$-definable if and only if  there exist a positive integer $l$, $\mathfrak R$-definable subsets $Y_1,\ldots, Y_l$ and subfamilies $\mathcal C_i$ of $\mycon(Y_i)$ such that $X=\bigcup_{i=1}^l\bigcup_{S \in \mathcal C_i}S$.
	
	In particular, $\mathfrak R^\natural$ is d-minimal and, for every $\mathfrak R^\natural$-definable set $X$ and a subfamily $\mathcal C$ of $\mycon(X)$, $\bigcup_{S \in \mathcal C}S$ is $\mathfrak R^\natural$-definable. 
\end{theorem}
\begin{proof}
	Let $\mathcal S_n$ be the collection of subsets $X$ of $\mathbb R^n$ such that  there exist a positive integer $l$, $\mathfrak R$-definable subsets $Y_1,\ldots, Y_l$ and subfamilies $\mathcal C_i$ of $\mycon(Y_i)$ with $X=\bigcup_{i=1}^l\bigcup_{S \in \mathcal C_i}S$.
	Observe that every $\mathfrak R$-definable subset of $\mathbb R^n$ belongs to $\mathcal S_n$.
	We may assume that $Y_1,\ldots Y_l$ are multi-cells for the following reason:
	
	We may assume $l=1$.
	We can decompose $Y_1$ into finitely many multi-cells $C_1,\ldots, C_k$ by assumptions (a,b) and Theorem \ref{thm:decomposition_into_multicells}.
	For every $1 \leq i \leq k$, $D \in \mycon(C_i)$ and $S \in \mathcal C_1$, we have either $D \subseteq S$ or $D \cap S=\emptyset$ because $D$ and $S$ are connected.
	Put $\mathcal D_i=\{D \in \mycon(C_i)\;|\; D \subseteq S \text{ for some } S \in \mathcal C_1\}$.
	We have $X=\bigcup_{i=1}^k \bigcup_{D \in \mathcal D_i}D$.
	By considering $C_1,\ldots, C_k$ and $\mathcal D_1,\ldots, \mathcal D_k$ instead of $Y_1$ and $\mathcal C_1$, we may assume that $Y_1,\ldots, Y_l$ are multi-cells.

	It is obvious that every member of $\mathcal S_n$ is $\mathfrak R^\natural$-definable.
	We show that every $\mathfrak R^\natural$-definable subset $X$ of $\mathbb R^n$ belongs to $\mathcal S_n$.
	By induction on the complexity of the formula defining $X$, we have only to prove the following:
	\begin{enumerate}
		\item[(i)] If $A, B \in \mathcal S_n$, we have $A \cap B \in \mathcal S_n$;
		\item[(ii)] If $A \in \mathcal S_n$, we have $\mathbb R^n \setminus A \in \mathcal S_n$.
		\item[(iii)] If $A \in \mathcal S_n$, we have $\pi(A) \in \mathcal S_{n-1}$, where $\pi:\mathbb R^n \to \mathbb R^{n-1}$ denotes the projection forgetting the last coordinate.
	\end{enumerate}
	By the definition of $\mathcal S_n$, there exist a positive integer $l_A$, $\mathfrak R$-definable subsets $Y_{A1},\ldots, Y_{Al_A}$ and subfamilies $\mathcal C_{Ai}$ of $\mycon(Y_{Ai})$ with $A=\bigcup_{i=1}^{l_A}\bigcup_{S \in \mathcal C_{Ai}}S$.
	We define $l_B$, $Y_{B1},\ldots, Y_{Bl_B}$ and $\mathcal C_{Bi}$ in the same manner.
	
	We prove clause (i).
	We can easily reduce to the case where $l_A=l_B=1$.
	Let $D$ be an element of $\mycon(Y_{A1} \cap Y_{B1})$.
	For every $S \in \mycon(Y_{A1}) \cup \mycon(Y_{B1})$, we have either $D \subseteq S$ or $D \cap S=\emptyset$ because $D$ and $S$ are connected, $D \subseteq Y_{A1}$ and $D \subseteq Y_{B1}$.
	Put $\mathcal C=\{D \in \mycon(Y_{A1} \cap Y_{B1})\;|\; \exists S_A \in \mathcal C_{A1} \ \exists S_B \in \mathcal C_{B1} \ D \subseteq S_A \cap S_B \}$.
	We have $A \cap B=\bigcup_{D \in \mathcal C}D$, which implies $A \cap B \in \mathcal S_n$.

	We prove clause (ii).
	By clause (i), we may assume $l_A=1$.
	We have $\mathbb R^n \setminus A = (\mathbb R^n \setminus Y_{A1}) \cup \bigcup_{S \in \mycon(Y_{A1}) \setminus \mathcal C_{A1}}S$.
	This deduces that $\mathbb R^n \setminus A \in \mathcal S_n$.

	We prove clause (iii).
	We may assume $l_A=1$ and $Y_{A1}$ is a multi-cell.
	By the definition of multi-cells, for every connected component of $D$ of $Y_{A1}$, $\pi(D)$ is a connected component of $\pi(Y_{A1})$.
	Put $\mathcal E=\{E \in \mycon(\pi(Y_{A1}))\;|\; E=\pi(D) \text{ for some }D \in \mathcal C_{A1}\}$.
	We have $\pi(A)=\bigcup_{E \in \mathcal E}E$, which implies $\pi(A) \in \mathcal S_{n-1}$.

	The remaining task is to show the `in particular' part.
	We only prove that $\mathfrak R^\natural$ is d-minimal.
	The remaining part is easy, and we omit the proof.
	For every subset $S$ of $\mathbb R^n$ and $x \in \mathbb R^{n-1}$, we denote $\{y \in \mathbb R\;|\; (x,y) \in S\}$ by $S_x$.
	Let $X$ be an $\mathcal R^\natural$-definable subset of $\mathbb R^n$.
	We show that there exists a positive integer $m$ such that $X_x$ is a union of an open set and at most $m$ discrete sets for every $x \in \mathbb R^{n-1}$.
	There exist a positive integer $l$, $\mathfrak R$-definable subsets $Y_1,\ldots, Y_l$ and subfamilies $\mathcal C_i$ of $\mycon(Y_i)$ with $X=\bigcup_{i=1}^l\bigcup_{S \in \mathcal C_i}S$.
	We can easily reduce to the case where $l=1$ and, furthermore, we may assume that $Y_1$ is a multi-cell by Theorem \ref{thm:decomposition_into_multicells}.
	By the definition of multi-cells, $(Y_1)_x$ is either open or discrete for every $x \in \mathbb R^{n-1}$.
	$X_x=\bigcup_{S \in \mathcal C_1}S_x$ is also either open or discrete.
\end{proof}

\begin{remark}\label{rem:almost}
	An expansion $\mathfrak R=(R,<,\ldots)$ of dense linear order without endpoints is called \textit{almost o-minimal} if every bounded definable subset of $R$ is a union of a finite set and finitely many open intervals.
	Let $\mathfrak R=(R,<,+,\ldots)$ be an almost o-minimal expansion of an ordered group.
	Note that we do not assume $R=\mathbb R$ here and every locally o-minimal structure whose underlying space is $\mathbb R$ is almost o-minimal by \cite[Proposition 4.11(1)]{Fuji_almost}.
	
	We slightly modify the definition of $\mathfrak R^\natural$ as follows: 
	 $\mathfrak R^\natural$ denotes the expansion of $\mathfrak R$ generated by all the sets of the form $\bigcup_{S \in \mathcal C}S$ for an $n \in \mathbb N$, an $\mathfrak R$-definable subset $X$ of $R^n$ and a subfamily $\mathcal C$ of $\mycon_{\text{semi}}(X)$.
	Here, $\mycon_{\text{semi}}(X)$ denotes the collection of semi-definable connected components.
	Semi-definably connected components of a definable set are defined in \cite[Theorem 3.6]{Fuji_almost}.
	
	A decomposition theorem into multi-cells holds for almost o-minimal expansion of an ordered group \cite[Theorem 4.22]{Fuji_almost}.
	Using this instead of Theorem \ref{thm:decomposition_into_multicells}, we can prove the following assertion in the same manner as Theorem \ref{thm:natural}:
	
	A subset $X$ of $R^n$ is $\mathfrak R^\natural$-definable if and only if there exist a positive integer $l$, $\mathfrak R$-definable subsets $Y_1,\ldots, Y_l$ and subfamilies $\mathcal C_i$ of $\mycon_{\text{semi}}(Y_i)$ such that $$X=\bigcup_{i=1}^l\bigcup_{S \in \mathcal C_i}S.$$
	In particular, for every $\mathfrak R^\natural$-definable set $X$ and a subfamily $\mathcal C$ of $\mycon(X)$, $\bigcup_{S \in \mathcal C}S$ is $\mathfrak R^\natural$-definable. 
		
	This assertion implies that $\mathfrak R^\natural$ is almost o-minimal.
\end{remark}


\begin{thebibliography}{99}
	
\bibitem{DMST}
A. Dolich, C. Miller, A. Savatovsky and A. Thamrongthanyalak, 
\emph{Connectedness in structures on the real numbers: o-minimality and undecidability},
J. Symbolic Logic, \textbf{87} (2022) 1243-1259.

\bibitem{vdD_dmin}
L. van den Dries, 
\emph{The field of reals with a predicate for the power of two},
Manuscripta Math., \textbf{54} (1985), 187--195.
	

\bibitem{Fuji_almost}
M. Fujita, 
\emph{Almost o-minimal structures and $\mathfrak X$-structures}, 
Ann. Pure Appl. Logic, \textbf{173} (2022), 103144.

\bibitem{Fuji_tame}
M. Fujita,
\emph{Locally o-minimal structures with tame topological properties},
J. Symbolic Logic, \textbf{88} (2023), 219--241.


\bibitem{Miller-dmin}
C. Miller,
\emph{Tameness in expansions of the real field},
In M. Baaz, S. -D. Friedman and J. Kraj\'{i}\u{c}ek eds., 
Logic Colloquium \textquotesingle 01,
Cambridge University Press (2005),
281--316.

\bibitem{Miller-choice}
C. Miller,
\emph{Definable choice in d-minimal expansions of ordered groups},
unpublished note available at \texttt{https://people.math.osu.edu/miller.1987/eidmin.pdf} (2006)

\bibitem{Sav}
A. Savatovsky,
\emph{Structure theorems for d-minimal expansions of the real additive ordered group and some consequences},
PhD Thesis, Universit\"at Konstanz (2020).

\bibitem{T}
A. Thamrongthanyalak, 
\emph{Michael's selection theorem in d-minimal expansions of the real field},
Proc. Amer. Math. Soc., \textbf{147} (2019), 1059-1071.

\end{thebibliography}
\end{document}